\documentclass[10pt, reqno]{amsart}
\usepackage{amsaddr}
\usepackage{latexsym,amsfonts}
\usepackage{amssymb,amsthm,upref,amscd}
\usepackage[T1]{fontenc}
\usepackage{times}
\usepackage{amscd}
\usepackage{enumerate}
\usepackage{mathrsfs}
\usepackage{amsmath}
\usepackage{color}
\usepackage{soul}
\usepackage{indentfirst}
\usepackage{comment}
\PassOptionsToPackage{reqno}{amsmath}

\newtheorem{thm}{Theorem}[section]
\newtheorem{prop}{Proposition}[section]
\newtheorem{defi}{Definition}[section]
\newtheorem{lem}{Lemma}[section]

\newtheorem{rem}{Remark}[section]
\theoremstyle{notation}

\newcommand{\R}{\mathbb{R}}

\numberwithin{equation}{section}

\newcommand{\eps}{\epsilon}

\newcommand{\wto}{\rightharpoonup}

\makeatletter \@addtoreset{equation}{section} \makeatother

\usepackage[textwidth=138mm,
textheight=215mm,
left=30mm,
right=30mm,
top=25.4mm,
bottom=25.4mm,
headheight=2.17cm,
headsep=4mm,
footskip=12mm,
heightrounded,]{geometry}

\usepackage[colorlinks=true,urlcolor=blue,
citecolor=black,linkcolor=black,linktocpage,pdfpagelabels,
bookmarksnumbered,bookmarksopen]{hyperref}
\newcounter{const}
\title[Double Phase Eigenvalue Problems]{Non-autonomous double phase eigenvalue problems with indefinite weight and lack of compactness} 

\date{\today}

\thanks{{\it Acknowledgments}. T. Gou was supported by the National Natural Science Foundation of China (No. 12101483) and the Postdoctoral Science Foundation of China (No. 2021M702620). V.D. R\u adulescu was supported by a grant of the Romanian Ministry of Research, Innovation, and Digitization, CNCS/CCCDI-UEFISCDI (No. PCE 137/2021), within PNCDI III. The authors would like to thank warmly the anonymous referees for his/her very precise reading of our manuscript and for giving constructive comments and suggestions.}

\author[T. Gou]{{Tianxiang Gou}}
\address{School of Mathematics and Statistics, Xi'an Jiaotong University,\\
	Xi'an, Shaanxi 710049, China}
\email{tianxiang.gou@xjtu.edu.cn}

\author[V.D.\ R\u adulescu]{{Vicen\c tiu D. R\u adulescu}}
\address{Faculty of Applied Mathematics, AGH University of Science and Technology,\\
\qquad\qquad\qquad	al. Mickiewicza 30, 30-059
	Krakow, Poland \newline\indent Brno University of Technology, Faculty of Electrical Engineering and Communication, \\
	\qquad\qquad\qquad Technick\'a 3058/10, Brno,
	61600, Czech Republic
	\newline\indent
	Department of Mathematics, University of
	Craiova, 200585 Craiova, Romania 
}
\email{radulescu@inf.ucv.ro}

\begin{document}

\begin{abstract}
In this paper, we consider eigenvalues to the following double phase problem with unbalanced growth and indefinite weight,
$$
-\Delta_p^a u-\Delta_q u =\lambda m(x) |u|^{q-2}u \quad \mbox{in} \,\, \R^N,
$$
where {$N \geq 2$}, {$1<p, q<N$, $p \neq q$}, ${a \in C^{0, 1}(\R^N, [0, +\infty))}$, $a \not\equiv 0$ and $m: \R^N \to \R$ is {an indefinite sign weight which may admit nontrivial positive and negative parts}. Here $\Delta_q$ is the $q$-Laplacian operator and $\Delta_p^a$ is the weighted $p$-Laplace operator defined by $\Delta_p^a u:=\textnormal{div}(a(x) |\nabla u|^{p-2} \nabla u)$. The problem can be degenerate, in the sense that the infimum of $a$ in $\R^N$ may be zero. Our main results distinguish between the cases $p<q$ and $q<p$. In the first case, we establish the existence of a {\it continuous} family of eigenvalues, starting from the principal frequency of a suitable single phase eigenvalue problem. In the latter case, we prove the existence of a {\it discrete} family of positive eigenvalues, which diverges to infinity.


\end{abstract}

\maketitle

\thispagestyle{empty}

\section{Introduction}

In this paper, we investigate eigenvalues to the following double phase problem with unbalanced growth and indefinite weight,
\begin{align} \label{eque}
-\Delta_p^a u-\Delta_q u =\lambda m(x) |u|^{q-2}u \quad \mbox{in} \,\, \R^N,
\end{align}
where {$N \geq 2$}, {$1<p, q<N$, $p \neq q$}, ${a \in C^{0, 1}(\R^N, [0, +\infty))}$, $a \not\equiv 0$ and $m: \R^N \to \R$ is {an indefinite sign weight which may admit nontrivial positive and negative parts}. Here $\Delta_q$ is the $q$-Laplacian operator and $\Delta_p^a$ is the weighted $p$-Laplace operator defined by $\Delta_p^a u:=\textnormal{div}(a(x) |\nabla u|^{p-2} \nabla u)$. Throughout of this paper, we shall always assume that the weight function $m : \R^N \to \R$ satisfies the following assumption, 
\begin{itemize}
\item[$(H)$] $m=m_1-m_2$, where $m_1, m_2\geq 0$, $m_1 \not \equiv 0$, $m_1 \in L^{\frac Nq}(\R^N) \cap L^{\infty}(\R^N)$ and $m_2 \in L^{\infty}(\R^N)$.
\end{itemize}

\begin{rem}
In our case, $m_2=0$ is allowable.
\end{rem}

Problems like \eqref{eque} arise when one looks for the stationary solutions of reaction-diffusion
systems of the form
$$
u_t=\mbox{div}\,[D(x,\nabla u)\nabla u]+g(x,u)\quad (x,t)\in\R^N\times (0,\infty),
$$
where $D(x,\nabla u)=a(x)|\nabla u|^{p-2}+|\nabla u|^{q-2}$. This system has a wide range of applications in physics
and related fields, such as biophysics, plasma physics, and chemical reaction design (see
\cite{p12, Si}). In such applications, the function $u$ is a state variable and describes density or concentration
of multi-component substances, $\mbox{div}\,[D(x,\nabla u)\nabla u]$ corresponds to the diffusion with a
diffusion coefficient $D(x,\nabla u)$, and $g(x, u)$ is the reaction and relates to source and loss processes.
Typically, in chemical and biological applications, the reaction term $g(x, u)$ has a polynomial
form with respect to the unknown concentration denoted by $u$.

The analysis of the double phase eigenvalue problem \eqref{eque} is closely associated with the following single phase quasilinear eigenvalue problem,
\begin{align} \label{initial}
-\Delta_r^a u =\mu m(x)|u|^{r-2}u \quad \mbox{in} \,\, \R^N.
\end{align}
The first part of the paper is devoted to the study of \eqref{initial}. The main results we establish regarding \eqref{initial} are upcoming Theorem \ref{thm0} and Proposition \ref{nodal}, which reveal that there exist a sequence of eigenvalues to \eqref{initial} and the first eigenvalue is simple. In the case of bounded domains and $r=2$, this problem is related to the Riesz-Fredholm theory of self-adjoint and compact operators. The anisotropic linear case (if $r=2$ and $m(\cdot)$ is non-constant) was first considered in the pioneering papers of Bocher \cite{p1}, Hess and Kato \cite{p2} and Pleijel \cite{p3}. An important contribution in the case of unbounded domains is due to Allegretto and Huang \cite{AH} and Szulkin and Willem \cite{p4}. In \cite{p4}, the authors assumed that weight function may have singular points.

Equation \eqref{eque} contains the contribution of two differential operators in the left-hand side, so this problem is not homogeneous. In fact, the differential operator $u\mapsto -\Delta_p^a u-\Delta_q u$ is related to the  ``double-phase" variational functional defined by
$$
u\mapsto \int_{\R^N} a(x)|\nabla u|^p+|\nabla u|^q \,dx.
$$
The integrand of this functional is the function
$$
\xi(x,t)=a(x)t^p+t^q\ \ \mbox{for all}\ x\in\R^N \mbox{and}\ t\geq 0.
$$
When $a\equiv 1$, then \eqref{eque} becomes the so-called $p$ \& $q$ Laplacian problem, which was investigated by Benouhiba and Belyacine \cite{BB1, BB}. A feature of the present paper is that we do not assume that the function $a(\cdot)$ is bounded away from zero, that is, we do not require that $\displaystyle{\rm essinf}_{x\in \R^N}a(x)>0$. This implies that the integrand $\xi(x,t)$ exhibits unbalanced growth, namely there holds that
\begin{align} \label{ub}
t^q\leq\xi(x,t)\leq C_0(t^p+t^q)\ \mbox{for all $x\in\R^N$ and $t\geq 0$},
\end{align}
where $C_0>0$ is a constant. In this scenario, the study is carried out in the framework of Musielak-Orlicz-Sobolev spaces. Such functionals were first investigated by Marcellini \cite{marce1,marce2,marce3} in the context of problems of the calculus of variations and of nonlinear elasticity for strongly anisotropic materials. For such problems, there is no global (that is, up to the boundary) regularity theory. There are only interior regularity results, which are primarily due to Baroni et al. \cite{BCM} and {Marcellini \cite{CMM, marce3, marce4}}. In fact, most of works dealt with double phase problems having unbalanced growth in bounded domains of $\R^N$, we refer the readers to \cite{GP, GW, GLL, PPR, PRZ, PVV} and references therein. However, there exist relatively few ones treating the problems in $\R^N$. The study of eigenvalue problems like \eqref{eque} is open until now. Since \eqref{eque} is set in the whole space $\R^N$, then lack of compactness is one of major difficulties we encounter to discuss the eigenvalue problem \eqref{eque} in Musielak-Orlicz-Sobolev spaces and more careful analysis is needed in suitable weighted functions spaces. {Indeed, this is mainly because the embedding $W^{1, \xi}(\R^N)\hookrightarrow L^r(\R^N)$ is only continuous for any $q \leq r \leq q^*$ (see Lemma \ref{embedding2}) and the weight function $m: \R^N \to \R$ is indefinite, which cause that the verification of the compactness of the underlying (minimizing and Palasi-Smale) sequences becomes difficult. Consequently, we manage to study the problem \eqref{eque} in a new weighted Sobolev space $E$ defined by the completion of $C^{\infty}_0(\R^N)$ under the norm
$$
\|u\|_E:=\|\nabla u\|_{\xi} +\left(\int_{\R^N} |u|^q \max\{m_2, \omega\} \,dx\right)^{\frac 1q}, \quad \omega(x):=\frac{1}{(1+|x|)^q}, \quad x \in \R^N, 
$$ 
where $\|\cdot\|_{\xi}$ denotes the standard norm in $D^{1, \xi}(\R^N)$. Here $W^{1, \xi}(\R^N)$ and $D^{1, \xi}(\R^N)$ are Musielak-Orlicz-Sobolev spaces defined in Section 2.} In the present paper, when $p<q$, we establish the existence of a continuous family of eigenvalues to \eqref{eque}, starting from the principal frequency to \eqref{initial}, see Theorems \ref{nonexist} and \ref{thm1}. While $q<p$, we prove the existence of a discrete family of positive eigenvalues to \eqref{eque}, which diverges to infinity, see Theorem \ref{thm2} and Proposition \ref{simple1}. The results we derive reveal new facts of eigenvalues to double phase problems in $\R^N$. {In both cases, we actually need to assume that $q<q^*:=\frac{Nq}{N-q}$, because of the unbalanced growth property \eqref{ub} with respect to the double phase operator and the dominance is the $q$-Laplacian term. Thus the problem under consideration is Sobolev subcritical and the energy functional $J$ corresponding to \eqref{eque} is well-defined in the Sobolev space $E$ by Theorem \ref{embedding2}, where
$$
J(u):=\frac 1p \int_{\R^N} a(x)|\nabla u|^p\,dx + \frac 1q \int_{\R^N} |\nabla u|^q \, dx-\frac{\lambda}{q} \int_{\R^N} m(x)|u|^q \,dx.
$$
Observe that $\frac{p}{q}<1+\frac 1 N$ implies that $p<q^*$. When double phase problems are set in bounded domains in $\R^N$, then the condition $\frac{p}{q}<1+\frac 1 N$ can be applied to prove the desired compact embedding results, for example \cite[Proposition 4]{PPR}. While double phase problems are set in $\R^N$, then the condition $\frac{p}{q}<1+\frac 1 N$ can no longer be applicable to derive the compact embedding results, which leads to lack of compactness for the study. In the present paper, such a condition is actually used to guarantee the regularity of solutions to \eqref{eque} (see \cite{CS, CM}), which along with the maximum principle developed in \cite{PRZ, PVV} can lead to the simplicity of eigenvalues, see Proposition \ref{simple1}. 
}

\section{Preliminaries}

In the section, we are going to present some preliminary results used to establish our main theorems. To deal with the eigenvalue problem \eqref{eque}, we shall work in the corresponding Musielak-Orlicz-Sobolev space. For the convenience of the readers, let us first present a few definitions from \cite[Section 2]{DHHR} concerning the main notions and function spaces used in this paper.

\begin{defi}
A function $\varphi : [0, +\infty] \to [0, +\infty)$ is called a $\Phi$-function if $\varphi$ is convex and left-continuous on $[0, +\infty)$. In addition, $\varphi$ satisfies that 
$$
\varphi(0)=0, \quad \lim_{t \to 0^+} \varphi(t)=0, \quad \lim_{t \to +\infty} \varphi(t)=+\infty.
$$
\end{defi}

\begin{defi}
A function $\xi :\R^N \times [0, +\infty] \to [0, +\infty)$ is called a generalized $\Phi$-function if it satisfies the following conditions:
\begin{itemize}
\item [$(\textnormal{i})$] for almost every $x \in \R^N$, $\xi(x, \cdot)$ is a $\Phi$-function;
\item [$(\textnormal{ii})$] for almost every $t \in [0, +\infty)$, $\xi(\cdot, t)$ is measurable.
\end{itemize}
\end{defi}

\begin{defi}
A generalized $\Phi$-function $\xi :\R^N \times [0, +\infty] \to [0, +\infty)$ satisfies $\Delta_2$-condition if there exists $K \geq 2$ such that, for almost every $x \in \R^N$ and $t \geq 0$,
$$
\xi(x, 2t) \geq K \xi(x, t).
$$
\end{defi}

\begin{defi}
A $\Phi$-function $\varphi : [0, +\infty] \to [0, +\infty)$ is said to be a $N$-function if it is continuous and positive on $[0, +\infty)$. In addition, it satisfies that
$$
\lim_{t \to 0^+} \frac{\varphi(t)}{t}=0, \quad \lim_{t \to +\infty} \frac{\varphi(t)}{t}=+\infty.
$$
A generalized $\Phi$-function $\xi :\R^N \times [0, +\infty] \to [0, +\infty)$ is said to be a generalized $N$-function if, for almost every $x \in \R^N$, $\xi(x, \cdot)$ is a $N$-function.
\end{defi}

\begin{defi}
A generalized $N$-function $\xi :\R^N \times [0, +\infty] \to [0, +\infty)$ is called uniformly convex if, for any $\eps>0$, there exists $\delta>0$ such that, for almost every $x \in \R^N$,
$$
\xi\left(x, \frac{s+t}{2}\right) \leq (1-\delta) \frac{\xi(x,s)+\xi(x, t)}{2},
$$
whenever $s, t \geq 0$ and $|x-t| \geq \eps \max \left\{|s|, |t|\right\}$.
\end{defi}

With these definitions in hand, we are now ready to introduce the double phase function $\xi: \R^N \times [0, +\infty) \to  [0, +\infty)$ corresponding to \eqref{eque} as 
\begin{align} \label{dpf}
\xi(x, t):=a(x) t^p +t^q, \quad x \in \R^N, \, t \geq 0.
\end{align} 
It is simple to check that $\xi$ is a generalized $N$-function. Moreover, $\xi$ is uniformly convex and it satisfies the $\Delta_2$-condition. Let us denote by $M(\R^N)$ the space consisting of all Lebesgue measurable function $u :\R^N \to \R$.  The Musielak-Orlicz space $L^{\xi}(\R^N)$ is defined by
$$
L^{\xi}(\R^N):= \left\{ u \in M(\R^N) : \rho_{\xi}(u)<+\infty \right\},
$$
where $\rho_{\xi}$ is the modular function given by
\begin{align} \label{nr}
\rho_{\xi}(u):=\int_{\R^N} \xi(x, |u|) \,dx =\int_{\R^N} a(x) |u|^p+|u|^q \, dx.
\end{align}
Here the space $L^{\xi}(\R^N)$ is equipped with the Luxemburg norm given by
\begin{align} \label{nxi}
\|u\|_{\xi}:=\inf \left\{\lambda>0 : \rho_{\xi}\left(\frac{u}{\lambda}\right) \leq 1\right\}.
\end{align}
Using the above properties satisfied by $\xi$, we can easily check that $L^{\xi}(\R^N)$ is a Banach space, which is also separable and reflexive. The Musielak-Orlicz-Sobolev space $W^{1, \xi}(\R^N)$ is defined by
$$
W^{1,\xi}(\R^N):=\left\{ u \in L^{\xi}(\R^N) : |\nabla u| \in L^{\xi}(\R^N)\right\}.
$$
Here the space $W^{1,\xi}(\R^N)$ is equipped with the norm
$$
\|u\|_{1, \xi}:=\|u\|_{\xi} + \|\nabla u\|_{\xi},
$$
where $\|\nabla u\|_{\xi}:=\||\nabla u|\|_{\xi}$. Clearly, $W^{1, \xi}(\R^N)$ is a separable, reflexive Banach space. Let us introduce the associated homogeneous Musielak-Orlicz-Sobolev $D^{1,\xi}(\R^N)$ as the completion of $C_0^{\infty}(\R^N)$ under the norm $\|\nabla u\|_{\xi}$.


Next we are going to show some relations between the norm in $L^{\xi}(\R^N)$ and the modular function $\rho_{\xi}$ given by \eqref{nr} and \eqref{nxi} respectively, proofs of which can be completed by using the ingredients presented in \cite[Section 3.2]{HH}. 

\begin{lem} \label{order}
Let $\xi : \R^N \times [0, +\infty) \to [0, +\infty)$ be defined by \eqref{dpf}. Then the following assertions hold.
\begin{itemize}
\item[$(\textnormal{i})$] $\|u\|_{\xi}=\lambda $ if and only if $\rho_{\xi} \left( \frac{u}{\lambda}\right)=1$.
\item[$(\textnormal{ii})$] $\|u\|_{\xi}<1(=1, >1, respectively)$ if and only if $\rho_{\xi}(u)<1 (=1, >1, respectively)$.
\item[$(\textnormal{iii})$] If $\|u\|_{\xi}<1$, then $\|u\|^{\max\{p, q\}}_{\xi} \leq \rho_{\xi}(u) \leq \|u\|_{\xi}^{\min\{p, q\}}$.
\item[$(\textnormal{iv})$] If $\|u\|_{\xi}>1$, then $\|u\|^{\min\{p, q\}}_{\xi} \leq \rho_{\xi}(u) \leq \|u\|_{\xi}^{\max\{p, q\}}$.
\item[$(\textnormal{v})$]  $\lim_{n \to +\infty} \|u_n\|_{\xi}=0 (+ \infty, respectively)$ if and only if $\lim_{n \to +\infty} \rho_{\xi}(u_n)=0(+\infty, respectively)$.
\end{itemize} 

\end{lem}

Note that $t^q \leq \xi(x,t)$ for any $x \in \R^N$ and $t \in \R$, by the assertion $(\textnormal{ii})$ of Lemma \ref{order}, then there holds the following embedding result.

\begin{lem} \label{embedding1}
Let $\xi : \R^N \times [0, +\infty) \to [0, +\infty)$ be defined by \eqref{dpf}.
Then the embedding $L^{\xi}(\R^N) \hookrightarrow L^q(\R^N)$ is continuous. 
\end{lem}

As a consequence of Lemma \ref{embedding1} and Sobolev's embeddings in $W^{1,q}(\R^N)$ and $D^{1,q}(\R^N)$ for $1<q<N$, we have the following embedding result.

\begin{lem} \label{embedding2}
Let $\xi : \R^N \times [0, +\infty) \to [0, +\infty)$ be defined by \eqref{dpf}. Then the embedding $W^{1, \xi}(\R^N)\hookrightarrow W^{1,q}(\R^N) \hookrightarrow L^r(\R^N)$ is continuous for any $q \leq r \leq q^*$. Moreover, the embedding $D^{1,\xi}(\R^N) \hookrightarrow D^{1,q}(\R^N) \hookrightarrow L^{q^*}(\R^N)$ is continuous.
\end{lem}

\section{Main results}

In this section, we shall consider the eigenvalue problem \eqref{eque} under the assumption $(H)$. The hypothesis $(H)$ is always assumed to hold in what follows. First we shall present some results related to the following eigenvalue problem,
\begin{align} \label{eque1}
-\Delta_r^a u =\mu m(x)|u|^{r-2}u \quad \mbox{in} \,\, \R^N.
\end{align}

\begin{thm}  \label{thm0} 
Assume that $(H)$ holds, ${N \geq 2}$, ${1<r<N}$, ${a \in C^{0, 1}(\R^N, [0, +\infty))}$ and $a \not\equiv 0$. Then there exists a sequence of solutions $(\mu_{a,r,k}, u_{a,r,k}) \in \R \times D^{1, \eta}(\R^N)$ to \eqref{eque1} with $u_{a,r,k} \in \mathcal{M}$ and 
$$
0<\mu_{a,r,1}<\mu_{a,r,2} \leq \cdots \leq \mu_{a,r,k} \leq \cdots, \quad \lim_{k \to \infty} \mu_{a,r,k} \to +\infty \quad \mbox{as} \,\, k \to +\infty,
$$ 
where $\eta(x,t)=a(x)t^r$ for $x \in \R^N$ and $t \geq 0$,
$$
\mathcal{M}_r:=\left\{u \in D^{1, \eta}(\R^N) : \int_{\R^N} m(x)|u|^r \, dx =1\right\}.
$$
\end{thm}
\begin{proof}
Define 
$$
\Psi(u):=\int_{\R^N} a(x) |\nabla u|^r\, dx, \quad M_r:=\mathcal{M}_r \cap \mathcal{V},
$$
where the Sobolev space $\mathcal{V}$ is the completion of $C_0^{\infty}(\R^N)$ under the norm
$$
{\|\nabla u\|_{\eta} +\left(\int_{\R^N} |u|^r \max\{m_2, \omega\} \,dx\right)^{\frac 1r}, \quad \omega(x)=\frac{1}{(1+|x|)^r}, \quad x \in \R^N}.
$$
Reasoning as the proof of \cite[Lemma 1]{AH}, we are able to show that $\Psi(u)$ restricted on $M_r$ satisfies the Palais-Smale condition. Then, by adapting Ljusternik-Schnirelman theory as the proof of forthcoming Theorem \ref{thm2}, we can derive the desired conclusion. Thus the proof is completed.
\end{proof}



\begin{prop} \label{nodal}
Assume that $(H)$ holds, {$N \geq 2$}, ${1<r<N}$, ${a \in C^{0, 1}(\R^N, [0, +\infty))}$ and $a \not\equiv 0$. Then the first eigenvalue $\mu_{a,r,1}$ obtained in Theorem \ref{thm0} is simple and the eigenfunction $u_{a,r,1}$ has constant sign. Moreover, if $u \in D^{1,\eta}(\R^N)$ is a nontrivial solution to \eqref{eque1} corresponding to $\mu>\mu_{a,r,1}$, then $u$ is sign-changing.
\end{prop}

Since the function $m$ is an indefinite sign weight, then proof of Proposition \ref{nodal} is not straightforward. To prove this, we need the following auxiliary result.

\begin{lem} \label{Iuv1}
Define
$$
I(u, v):=-\int_{\R^N} \left(\Delta_r^a u \right) \frac{u^r-v^r}{u^{r-1}} \,dx - \int_{\R^N} \left(\Delta_r^a v \right) \frac{v^r-u^r}{v^{r-1}} \,dx, \quad u, v \in D^{1, \eta}(\R^N), u, v>0.
$$
Then $I(u, v) \geq 0$. Moreover, $I(u ,v)=0$ if and only if $u= kv$ for some $k \in \R$. 
\end{lem}
\begin{proof}
Observe that
$$
\nabla \left( \frac{u^r-v^r}{u^{r-1}}\right)= \left(1+(r-1) \left(\frac v u\right)^r\right) \nabla u-r \left(\frac v u\right)^{r-1} \nabla v,
$$
$$
\nabla \left( \frac{v^r-u^r}{v^{r-1}}\right)= \left(1+(r-1) \left(\frac u v\right)^r\right) \nabla v-r \left(\frac u v\right)^{r-1} \nabla u.
$$
Then, by the divergence theorem, we see that
\begin{align} \label{Iuvn}
\begin{split}
I(u ,v)& =\int_{\R^N} a(x)\left(\left(1+(r-1) \left(\frac v u\right)^r\right) |\nabla u|^r-r \left(\frac v u\right)^{r-1} |\nabla u|^{r-2} \left( \nabla v \cdot \nabla u\right)\right) \,dx \\
& \quad + \int_{\R^N} a(x)\left(\left(1+(r-1) \left(\frac u v\right)^r\right) |\nabla v|^r-r \left(\frac u v\right)^{r-1} |\nabla v|^{r-2} \left( \nabla u \cdot \nabla v\right)\right) \,dx.
\end{split}
\end{align}
Using Young's inequality, we have that
$$
r \left(\frac v u\right)^{r-1} |\nabla u|^{r-2} \left( \nabla v \cdot \nabla u\right) \leq r \left(\frac v u\right)^{r-1} |\nabla u|^{r-1} |\nabla v|\leq (r-1) \left(\frac v u\right)^r |\nabla u|^r + |\nabla v|^r,
$$
$$
r\left(\frac u v\right)^{r-1} |\nabla v|^{r-2} \left( \nabla u \cdot \nabla v\right) \leq r \left(\frac u v\right)^{r-1} |\nabla v|^{r-1} |\nabla u|\leq (r-1) \left(\frac u v\right)^r |\nabla v|^r + |\nabla u|^r.
$$
As a consequence, coming back to \eqref{Iuvn}, we can conclude that $I(u, v) \geq 0$. If $I(u, v)=0$, then 
$$
\nabla u \cdot \nabla v=|\nabla u| |\nabla v|, \quad \left(\frac vu \right)^r |\nabla u|^r=|\nabla v|^r, \quad \left(\frac uv\right)^r |\nabla v|^r=|\nabla u|^r,
$$
It then follows that
$$
| u \nabla v-v\nabla u|=0.
$$
This implies that there exists $k \in \R$ such that $u=k v$ and the proof is completed.
\end{proof}

\begin{proof}[Proof of Proposition \ref{nodal}]
Note first that
$$
\mu_{a,r, 1}=\inf_{u \in \mathcal{M}_r} \Psi(u).
$$
If $u \in \mathcal{M}_r$ satisfies that $\Psi(u)=u_{a,r,1}$, then $|u| \in \mathcal{M}_r$ and $\Psi(|u|)=u_{a,r,1}$. Therefore, without restriction, we may assume that $u_{a,r,1}$ is nonnegative. Observe that $u_{a,r,1} \in D^{1,\eta}(\R^N)$ satisfies the equation
$$
-\Delta_r^a u_{a,r,1} +\mu_{a,r,1} m_2(x)|u_{a,r,1}|^{r-2}u_{a,r,1}=\mu_{a,r,1} m_1(x)|u_{a,r,1}|^{r-2}u_{a,r,1} \geq 0 \quad \mbox{in} \,\, \R^N.
$$
By maximum principle, then $u_{a,r,1}>0$. Let $u_{a,r,1} \in \mathcal{M}_r$ and $v_{a,r,1} \in \mathcal{M}_r$ be two positive eigenfunctions corresponding to $\mu_{a,r,1}$, then
$$
-\Delta_r^a u_{a,r,1}=\mu_{a,r,1} m(x)u_{a,r,1}^{r-1},\quad -\Delta_r^a v_{a,r,1}=\mu_{a,r,1} m(x)v_{a,r,1}^{r-1} \quad \mbox{in} \,\, \R^N.
$$
It is simple to calculate that $I(u_{a,r,1}, v_{a,r,1})=0$. As a result of Lemma \ref{Iuv1}, we have that $u_{a,r,1}=k v_{a,r,1}$ for some $k \in \R$. This indicate that $\mu_{a,r,1}$ is simple.

Arguing by contradiction, we suppose that $u \in D^{1,\eta}(\R^N)$ is a nonnegative solution to \eqref{eque1} corresponding to $\mu>\mu_{a,r,1}$. By the maximum principle, then $u>0$. Notice that
$$
\int_{\R^N} a(x)|\nabla u|^r \,dx =\mu \int_{\R^N} m(x) |u|^r \, dx>0.
$$
In addition, we know that if $u \in D^{1, \eta}(\R^N)$ is a solution to \eqref{eque1}, then $k u \in D^{1, \eta}(\R^N)$ is also a solution to \eqref{eque1} for any $k \in \R \backslash \{0\}$. Then, by scaling, we may assume that
\begin{align} \label{scaling}
0<\int_{\R^N} m(x) |u|^r \, dx<1.
\end{align}
Let $u_{a,r,1} \in \mathcal{M}$ and $u_{a,r,1}>0$ be an eigenfunction to \eqref{eque1} corresponding to $\mu_{a,r,1}$.
Then $u_{a,r,1}$ solves 
the equation
$$
-\Delta_r^a u_{a,r,1}=\mu_{a,r,1} m(x)|u_{a,r,1}|^{r-2}u_{a,r,1} \quad \mbox{in} \,\, \R^N.
$$
As a consequence of Lemma \ref{Iuv1} and \eqref{scaling}, we have thta
\begin{align*}
0 \leq I(u, u_{a,r,1})&=\mu \int_{\R^N} m(x) \left(u^r-u_{a,r,1}^r \right) \,dx+\mu_{a,r,1} \int_{\R^N} m(x) \left(u^r_{a,r,1}-u^r\right) \,dx\\
&=(\mu-\mu_{a,r,1})\int_{\R^N} m(x)|u|^r \,dx -(\mu-\mu_{a,r,1}) <0.
\end{align*}
This is impossible, hence $u$ is sign-changing and the proof is completed.
\end{proof}

\begin{thm} \label{nonexist}
Assume that $(H)$ holds, {$N \geq 2$}, {$1<p, q<N$, $p \neq q$}, ${a \in C^{0, 1}(\R^N, [0, +\infty))}$ and $a \not\equiv 0$. Then \eqref{eque} has no nontrivial solutions in $D^{1, \xi}(\R^N)$ for any $0 \leq \lambda \leq \mu_{1,q,1}$, where $\mu_{1,q,1}>0$ is the first eigenvalue to \eqref{eque1} with $a\equiv1$ and $r=q$.
\end{thm}
\begin{proof}
Let $u \in D^{1, \xi}(\R^N)$ be a solution to \eqref{eque} for some $0 \leq \lambda  \leq \mu_{1,q,1}$. Observe first that
\begin{align} \label{ne1}
\int_{\R^N} a(x) |\nabla u|^p \, dx +\int_{\R^N} |\nabla u|^q \,dx =\lambda \int_{\R^N} m(x) |u|^q \, dx.
\end{align}
This implies that $u=0$ if $\lambda =0$. Let us assume that $0 < \lambda<\mu_{1,q,1}$. Assume that $u \neq 0$, it then follows from \eqref{ne1} that
\begin{align} \label{ne2}
\int_{\R^N} m(x) |u|^q \, dx > 0.
\end{align}
In addition, since $\mu_{1,q,1}>0$ is the first eigenvalue to \eqref{eque1}, then
\begin{align} \label{ne3}
\int_{\R^N} |\nabla u|^q \,dx \geq \mu_{1,q,1} \int_{\R^N} m(x)|u|^q\, dx.
\end{align}
This along with \eqref{ne1} leads to
$$
\mu_{1,q,1} \int_{\R^N} m(x)|u|^q\, dx \leq \lambda \int_{\R^N} m(x)|u|^q\, dx.
$$
Using \eqref{ne2}, we then get that $u=0$. This is a contradiction. Next we assume that $\lambda=\mu_1$. In this case, by combining \eqref{ne1} and \eqref{ne3}, we obtain that
$$
\int_{\R^N} a(x) |\nabla u|^p \,dx \leq 0,
$$
hence $u=0$. Thus the proof is completed.
\end{proof}


\subsection{Case $p<q$}
In this case, to establish the existence of solutions to \eqref{eque}, we shall adapt some ideas from \cite{AH}. Let us first introduce the weight function 
$$
\omega(x)=\frac{1}{(1+|x|)^q}, \quad x \in \R^N.
$$
Let $E$ be the completion of $C_0^{\infty}(\R^N)$ under the norm
$$
\|u\|_E:=\|\nabla u\|_{\xi} +\left(\int_{\R^N} |u|^q \max\{m_2, \omega\} \,dx\right)^{\frac 1q}.
$$
It is standard to conclude that $E$ is a separable and reflexive Banach space. In order to prove the existence of solutions to \eqref{eque}, we shall define the associated energy functional $J: E \to \R$ by
$$
J(u):=\frac 1p \int_{\R^N} a(x)|\nabla u|^p\,dx + \frac 1q \int_{\R^N} |\nabla u|^q \, dx-\frac{\lambda}{q} \int_{\R^N} m(x)|u|^q \,dx.
$$


\begin{thm} \label{thm1}
Assume that $(H)$ holds, {$N \geq 2$}, $1<p<q<N$, ${a \in C^{0, 1}(\R^N, [0, +\infty))}$ and $a \not\equiv 0$. Then there exist positive solutions to \eqref{eque} for any $\lambda >\mu_{1,q,1}$.
\end{thm}

In this case, we find that $J$ is unbounded from below in $E$. Indeed, let $u_{1,q,1} \in D^{1, q}(\R^N)$ be an eigenfunction of \eqref{eque1} corresponding to $\mu_{1,q,1}$. We observe that
$$
J(tu_1)=\frac {t^p}{p} \int_{\R^N} a(x)|\nabla u_{1,q,1}|^p\,dx + \frac {t^q}{q} \left(1-\frac{\lambda}{\mu_{1,q,1}}\right)\int_{\R^N} |\nabla u_{1,q,1}|^q \, dx.
$$
Since $p<q$ and $\lambda>\mu_1$, then $J(tu_{1,q,1}) \to -\infty$ as $t \to +\infty$. In this situation, to seek for solutions to \eqref{eque}, we introduce the Nehari manifold
$$
\mathcal{N}:=\{ u \in E \backslash \{0\} : I(u)=0\},
$$
where
$$
I(u):=\int_{\R^N} a(x)|\nabla u|^p\,dx+\int_{\R^N} |\nabla u|^q\,dx-\lambda \int_{\R^N} m(x)|u|^q\,dx.
$$
Then we are able to define the minimization problem
\begin{align} \label{min}
m:=\inf_{u \in \mathcal{N}} J(u).
\end{align}
Obviously, any minimizer of \eqref{min} is a solution to \eqref{eque}.

\begin{proof} [Proof of Theorem \ref{thm1}] Let $\{u_n\} \subset \mathcal{N}$ be a minimizing sequence to \eqref{min}. Then $m=J(u_n)+o_n(1)$ and $I(u_n)=0$. Since $I(|u|) \leq I(u)=0$ for any $u \in \mathcal{N}$, then there exists a unique $0<t_{|u|} \leq 1$ such that $I(t_{|u|}|u|)=0$, where
$$
t_{|u|}=\left(\frac{\int_{\R^N} a(x)|\nabla |u||^p\,dx}{\lambda \int_{\R^N} m(x)|u|^q\,dx-\int_{\R^N} |\nabla |u||^q\,dx}\right)^{\frac{1}{q-p}}.
$$
Moreover, for any $u \in \mathcal{N}$, we see that
\begin{align} \label{mins}
J(u)=J(u)-\frac 1 q I(u)= \left(\frac 1p -\frac 1q\right) \int_{\R^N} a(x)|\nabla u|^p\,dx.
\end{align}
Therefore, for any $u \in \mathcal{N}$,  
$$
J(t_{|u|}|u|)=t_{|u|}^p\left(\frac 1p -\frac 1q\right) \int_{\R^N} a(x)|\nabla |u||^p\,dx \leq \left(\frac 1p -\frac 1q\right) \int_{\R^N} a(x)|\nabla u|^p\,dx =J(u).
$$
As a consequence, we shall assume that $\{u_n\} \subset \mathcal{N}$ is a nonnegative minimizing sequence to \eqref{min}. Otherwise, we can replace $\{u_n\}$ by $\{t_{|u_n|} |u_n|\}$ as a new minimizing sequence to \eqref{min}. 

First we are going to prove that $m>0$. It follows from \eqref{mins} that $m \geq 0$. Let us argue by contradiction that $m=0$. Then, by \eqref{mins}, we have that 
$$
\int_{\R^N} a(x)|\nabla u_n|^p\,dx=o_n(1).
$$
Let us first assume that
\begin{align} \label{a1}
\int_{\R^N} m(x)|u_n|^q\,dx=o_n(1).
\end{align}
Since $I(u_n)=0$, then there holds that
$$
\int_{\R^N} |\nabla u_n|^q\,dx=o_n(1).
$$
In this case, we set
\begin{align}\label{defvn}
v_n:=\frac{u_n}{\left(\int_{\R^N} m(x)|u_n|^q\,dx\right)^{\frac 1q}} \geq 0, \quad \forall \,\, n \in \mathbb{N}^+.
\end{align}
It is easy to see that $\{v_n\} \subset \mathcal{M}_q$. Since $I(u_n)=0$, then
\begin{align} \label{l1}
\int_{\R^N} a(x)|\nabla v_n|^p \,dx=\frac{\int_{\R^N} a(x) |\nabla u_n|^p \,dx}{\left(\int_{\R^N} m(x)|u_n|^q\,dx\right)^{\frac p q}}=\frac{1}{\left(\int_{\R^N} m(x)|u_n|^q\,dx\right)^{\frac p q-1}} \left(\lambda-\int_{\R^N} |\nabla v_n|^q \,dx\right).
\end{align}
In view of \eqref{a1} and \eqref{l1}, then
$$
\int_{\R^N} a(x)|\nabla v_n|^p \,dx=o_n(1).
$$
It then yields that
\begin{align} \label{c11}
\int_{\R^N} m(x)|v_n|^p \,dx\leq \frac{1}{\mu_{a,p,1}} \int_{\R^N} a(x)|\nabla v_n|^p \,dx =o_n(1).
\end{align}
Invoking H\"older's inequality, Sobolev's inequality and \eqref{c11}, we then get that
\begin{align*}
\int_{\R^N} m(x)|v_n|^q \, dx &\leq  \left(\int_{\R^N} m(x)|v_n|^p \,dx\right)^{\theta}  \left(\int_{\R^N} m(x)|v_n|^{q^*} \,dx\right)^{1-\theta}  \\
& \leq  \|m\|_{\infty}^{1-\theta}\left(\int_{\R^N} m(x)|v_n|^p \,dx\right)^{\theta}  \left(\int_{\R^N} |v_n|^{q^*} \,dx\right)^{1-\theta}  \\
& \leq C \|m\|_{\infty}^{1-\theta}\left(\int_{\R^N} m(x)|v_n|^p \,dx\right)^{\theta}  \left(\int_{\R^N} |\nabla v_n|^q \,dx\right)^{\frac{q^*(1-\theta)}{q}}=o_n(1),
\end{align*}
where $0<\theta<1$ and $q=\theta p +(1-\theta) q^*$. This is a contradiction, because of $v_n \in \mathcal{M}_q$.
Let us next assume that there exists some $\lambda_0>0$ such that
$$
\int_{\R^N} m(x)|u_n|^q\,dx=\lambda_0+o_n(1).
$$
Since $I(u_n)=0$, then
$$
\int_{\R^N} |\nabla v_n|^q\,dx=\lambda-\frac{\int_{\R^N} a(x) |\nabla u_n|^p \,dx}{\int_{\R^N} m(x)|u_n|^q\,dx}=\lambda+o_n(1).
$$
Therefore, there holds that $\|\nabla v_n\|_p^p=\lambda +o_n(1)$. In virtue of \cite[Lemma 1]{AH}, we then get that $\{v_n\}$ is compact in $V$, where the Sobolev space $V$ is the completion of $C_0^{\infty}(\R^N)$ under the norm
$$
\|\nabla u\|_q +\left(\int_{\R^N} |u|^q \max\{m_2, \omega\} \,dx\right)^{\frac 1q}.
$$
Let $v \in V$ be such that $v_n \to v$ in $V$ as $n \to \infty$, then $v \neq 0$ and $v \geq 0$. It then infers that $v \in V$ is a nonnegative eigenfunction to \eqref{eque1} corresponding to $\lambda$. By Lemma \ref{nodal}, we reach a contradiction, because of $\lambda>\mu_{1,q,1}$. 
As a consequence, we have that $m>0$. 

It is standard to show that $\mathcal{N}$ is a natural constraint. By the fact that there exists a nonnegative minimizing sequence to \eqref{min} and applying Ekeland's variational principle, then there exists a Palais-Smale sequence $\{u_n\} \subset E$ with $u_n^-=o_n(1)$ and $I(u_n)=o_n(1)$ for $J$ at the level $m>0$. Let us now prove that $\{u_n\}$ is bounded in $E$. Observe that 
\begin{align} \label{b1}
m+o_n(1)=J(u_n)-\frac 1q I(u_n)+o_n(1)=\left(\frac 1p-\frac 1q\right)\int_{\R^N} a(x) |u_n|^p \,dx.
\end{align}
Let us verify that $\{\|\nabla u_n\|_q\}$ is bounded. On the contrary, we may assume that $\|\nabla u_n\|_q \to + \infty$ as $n \to \infty$. Define $v_n$ by \eqref{defvn}, use the fact that $I(u_n)=o_n(1)$ and \eqref{b1}, then there holds that $\|\nabla v_n\|_p^p=\lambda +o_n(1)$. With the help of \cite[Lemma 1]{AH}, we can also reach a contradiction. This implies that $\{\|\nabla u_n\|_q\}$ is bounded. By Hardy's inequality, it then follows that
\begin{align*} 
\int_{\R^N} \frac{|u_n|^q}{(1+|x|)^q} \, dx\leq \left(\frac{p}{N-p}\right)^p \int_{\R^N} |\nabla u_n|^q \, dx \leq C.
\end{align*}
Notice that $I(u_n)=o_n(1)$, then
\begin{align*}
\int_{\R^N}m_2(x) |u_n|^q \,dx&=\int_{\R^N}m(x) |u_n|^q \,dx-\int_{\R^N}m_1(x) |u_n|^q \,dx \\
& \leq \int_{\R^N}m(x) |u_n|^q \,dx =\int_{\R^N} a(x)|\nabla u_n|^p\,dx+\int_{\R^N} |\nabla u_n|^q\,dx +o_n(1) \leq C.
\end{align*}
As a result, we get that $\{u_n\}$ is bounded in $E$. Then there exists $u \in E$ such that $u_n \wto u$ in $E$ as $n \to \infty$. Since $\{u_n\} \subset E$ is a bounded Palais-Smale sequence for $J$, then
\begin{align} \label{equ1}
-\Delta_p^a u_n-\Delta_q u_n =\lambda m(x)|u_n|^{q-2}u_n+o_n(1) \quad \mbox{in} \,\, \R^N.
\end{align}
Therefore, we are able to derive that $u \in E$ satisfies the following equation
\begin{align} \label{equ2}
-\Delta_p^a u-\Delta_q u = \lambda m(x)|u|^{q-2}u \quad \mbox{in} \,\, \R^N.
\end{align}
Since the embedding $E \hookrightarrow D^{1, \xi}(\R^N) \hookrightarrow L^{q^*}(\R^N)$ is continuous by Lemma \ref{embedding2}, then $\{u_n\}$ is bounded in $L^{q^*}(\R^N)$ and $u_n \wto u$ in $L^{q^*}(\R^N)$ as $n \to \infty$. It follows that $\{|u_n|^q\}$ is bounded in $L^{\frac{N}{N-q}}(\R^N)$ and $|u_n|^q \wto |u|^q$ in $L^{\frac{N}{N-q}}(\R^N)$ as $n \to \infty$. Due to $m_1 \in L^{\frac Nq}(\R^N)$, we have
\begin{align} \label{convm1}
\int_{\R^N} m_1(x)|u_n|^q \,dx=\int_{\R^N} m_1(x)|u|^q \,dx+o_n(1).
\end{align}
This readily indicates that $u \neq 0$. Otherwise, there holds that
\begin{align} \label{mz}
\int_{\R^N} m(x)|u_n|^q \,dx=\int_{\R^N} m_1(x)|u_n|^q \,dx-\int_{\R^N} m_2(x)|u_n|^q \,dx \leq o_n(1).
\end{align}
Since $I(u_n)=o_n(1)$, then 
$$
\int_{\R^N} a(x)|\nabla u_n|^p\,dx + \int_{\R^N} |\nabla u_n|^q \, dx \leq o_n(1).
$$
This in turn gives that $J(u_n)=o_n(1)$, which is impossible, because of $m>0$. Therefore, $u$ is a nontrivial solution to \eqref{eque}. Moreover, as a consequence of maximum principle, see \cite[Proposition 2.3]{PVV}, we have that $u>0$. Thus the proof is completed.
\end{proof}

\subsection{Case $q<p$}

Next we are going to deal with the case that $q<p$. In this case, we define
$$
\Phi(u):=\frac 1p \int_{\R^N} a(x)|\nabla u|^p\,dx + \frac 1q \int_{\R^N} |\nabla u|^q \, dx,  
$$
$$
\mathcal{S}:=\left\{ u \in E : \frac 1q \int_{\R^N} m(x)|u|^q \, dx=1 \right\}.
$$

\begin{lem} \label{psc}
Assume that $(H)$ holds, {$N \geq 2$}, $1<q<p<N$, ${a \in C^{0, 1}(\R^N, [0, +\infty))}$ and ${a \not\equiv 0}$. Then $\Phi$ restricted on $\mathcal{S}$ satisfies the Palais-Smale condition at any level $c \in \R$.
\end{lem}
\begin{proof}
Let $\{u_n\} \subset E$ be a Palais-Smale sequence for $\Phi$ restricted on $\mathcal{S}$ at the level $c \in \R$. Then
$$
\Phi(u_n)=c+o_n(1), \quad (\Phi_{\mid_{\mathcal{S}}})'(u_n)=o_n(1).
$$
The aim is to prove that $\{u_n\}$ is compact in $E$. It is straightforward to see that $\{u_n\}$ is bounded in $D^{1, \xi}(\R^N)$, because of $\{u_n\} \subset \mathcal{S}$. In virtue of Hardy's inequality, we obtain that
$$
\int_{\R^N} \frac{|u_n|^q}{(1+|x|)^q} \, dx \leq \left(\frac{p}{N-p}\right)^p \int_{\R^N} |\nabla u_n|^q \, dx \leq C.
$$
In addition, note that
$$
\int_{\R^N} m_2(x)|u_n|^q \,dx=\int_{\R^N} m_1(x)|u_n|^q \,dx-\int_{\R^N} m(x)|u_n|^q \,dx.
$$
By H\"older's inequality and Sobolev's inequality, we have
\begin{align*}
\int_{\R^N} m_1(x)|u_n|^q \,dx & \leq \left(\int_{\R^N}|m_1|^{\frac N q} \,dx \right)^{\frac qN} \left(\int_{\R^N}|u_n|^{\frac{qN}{N-q}} \, dx \right)^{\frac{N-q}{N}} \\
&\leq C \left(\int_{\R^N}|m_1|^{\frac N q} \,dx \right)^{\frac qN} \int_{\R^N}|\nabla u_n|^q \, dx \leq C.
\end{align*}
Accordingly, we obtain that $\{u_n\}$ is bounded in $E$. It then yields that there exists $u \in E$ such that $u_n \wto u$ in $E$ as $n \to \infty$. Since the embedding $E \hookrightarrow L^{q^*}(\R^N)$ is continuous, then $\{u_n\}$ is bounded in $L^{q^*}(\R^N)$ and $u_n \wto u$ in $L^{q^*}(\R^N)$ as $n \to \infty$. It then follows that $\{|u_n|^q\}$ is bounded in $L^{\frac{N}{N-q}}(\R^N)$ and $|u_n|^q \wto |u|^q$ in $L^{\frac{N}{N-q}}(\R^N)$ as $n \to \infty$. Due to $m_1 \in L^{\frac Nq}(\R^N)$, then
\begin{align} \label{convm1}
\int_{\R^N} m_1(x)|u_n|^q \,dx=\int_{\R^N} m_1(x)|u|^q \,dx+o_n(1).
\end{align}
It readily indicates that $u \neq 0$. Otherwise, by \eqref{convm1}, there holds that
\begin{align} \label{mz}
q=\int_{\R^N} m(x)|u_n|^q \,dx=\int_{\R^N} m_1(x)|u_n|^q \,dx-\int_{\R^N} m_2(x)|u_n|^q \,dx \leq o_n(1).
\end{align}
This is impossible.
Since $\{u_n\} \subset E$ is a bounded Palais-Smale sequence for $\Phi$ restricted on $\mathcal{S}$, then there exists a sequence $\{\lambda_n\} \subset \R$ such that $u_n \in E$ satisfies the equation
\begin{align} \label{equ1}
-\Delta_p^a u_n-\Delta_q u_n =\lambda_n m(x)|u_n|^{q-2}u_n+o_n(1) \quad \mbox{in} \,\, \R^N,
\end{align}
where 
$$
\lambda_n=\frac 1q \int_{\R^N} a(x)|\nabla u_n|^p\,dx + \frac 1q \int_{\R^N} |\nabla u_n|^q \, dx +o_n(1).
$$
Notice that $\{u_n\}$ is bounded in $D^{1,\xi}(\R^N)$, then $\{\lambda_n\}$ is bounded in $\R$ and there exists $\lambda \in \R$ such that $\lambda_n \to \lambda$ in $\R$ as $n \to \infty$. Furthermore,  $u \in E$ and it satisfies the equation
\begin{align} \label{equ2}
-\Delta_p^a u-\Delta_q u = \lambda m(x)|u|^{q-2}u \quad \mbox{in} \,\, \R^N.
\end{align}
Thanks to $u \neq 0$, we then have that$\lambda > 0$. Taking into account \eqref{equ1} and \eqref{equ2}, we conclude that
\begin{align} \label{compact} \nonumber
&\int_{\R^N} \left(a(x) \left(|\nabla u_n|^{p-2} \nabla u_n-|\nabla u|^{p-2} \nabla u \right)+\left(|\nabla u_n|^{q-2} \nabla u_n-|\nabla u|^{q-2} \nabla u \right)\right) \cdot \left(\nabla u_n-\nabla u\right)\,dx \\
&=\lambda_n \int_{\R^N}m(x)|u_n|^{q-2}u_n \left(u_n-u\right) \,dx-\lambda \int_{\R^N}m(x)|u|^{q-2}u \left(u_n-u\right) \,dx +o_n(1) \\ \nonumber
&=\left(\lambda_n-\lambda\right) \int_{\R^N}m(x)|u_n|^{q-2}u_n \left(u_n-u\right) \,dx+\lambda \int_{\R^N}m(x)\left(|u_n|^{q-2}u_n -|u|^{q-2}u\right)\left(u_n-u\right) \,dx+o_n(1).
\end{align}
Observe first that
\begin{align*}
\left|\int_{\R^N}m(x)|u_n|^{q-2}u_n \left(u_n-u\right) \,dx\right| &\leq \left|\int_{\R^N}m_1(x)|u_n|^{q-2}u_n \left(u_n-u\right) \,dx\right| \\
 &\quad +\left|\int_{\R^N}m_2(x)|u_n|^{q-2}u_n \left(u_n-u\right) \,dx\right|.
\end{align*}
In addition, we see that
\begin{align*}
\left|\int_{\R^N}m_1(x)|u_n|^{q-2}u_n \left(u_n-u\right) \,dx\right|  &\leq \left(\int_{\R^N}|m_1|^{\frac Nq} \,dx \right)^{\frac qN} \left(\int_{\R^N}|u_n|^{q^*} \,dx \right)^{\frac{q-1}{q^*}}\left(\int_{\R^N}|u_n-u|^{q^*} \,dx \right)^{\frac{1}{q^*}},
\end{align*}
\begin{align*}
\left|\int_{\R^N}m_2(x)|u_n|^{q-2}u_n \left(u_n-u\right) \,dx\right| &\leq \left(\int_{\R^N}m_2(x)|u_n|^q \,dx\right)^{\frac {q-1}{q}}\left(\int_{\R^N} m_2(x)|u|^q \,dx\right)^{\frac 1q} \\
& \quad +\int_{\R^N}m_2(x)|u_n|^q \,dx.
\end{align*}
Therefore, utilizing the fact that $\{u_n\}$ is bonded in $E$, we get that
$$
\left|\int_{\R^N}m(x)|u_n|^{q-2}u_n \left(u_n-u\right) \,dx\right| \leq C.
$$
It necessarily follows that
\begin{align} \label{s1}
\left(\lambda_n-\lambda\right) \int_{\R^N}m(x)|u_n|^{q-2}u_n \left(u_n-u\right) \,dx=o_n(1).
\end{align}
Note that $u_n \wto u$ in $E$ as $n \to \infty$, then $u_n \wto u$ in $D^{1,q}(\R^N)$ as $n \to \infty$. We then deduce that $u_n \to u$ in $L^q_{loc}(\R^N)$ as $n \to \infty$. As a consequence, we have that
\begin{align} \label{ss2}
\int_{B(0, R)}m(x)\left(|u_n|^{q-2}u_n -|u|^{q-2}u\right)\left(u_n-u\right) \,dx=o_n(1).
\end{align}
On the other hand, by H\"older's inequality and Sobolev's inequality, we get that
\begin{align} \label{s2}
\begin{split}
&\int_{\R^N \backslash B(0, R)}m(x)\left(|u_n|^{q-2}u_n -|u|^{q-2}u\right)\left(u_n-u\right) \,dx \\
&\leq\left(\int_{\R^N \backslash B(0, R)}|m_1|^{\frac Nq} \,dx \right)^{\frac qN}\left(\|u_n\|_{q^*}^{q-1}+\|u\|_{q^*}^{q-1}\right) \|u_n-u\|_{q^*} \\
& \leq C \left(\int_{\R^N \backslash B(0, R)}|m_1|^{\frac Nq} \,dx \right)^{\frac qN}\left(\|\nabla u_n\|_{q}^{q-1}+\|\nabla u\|_{q}^{q-1}\right) \|\nabla u_n-\nabla u\|_{q}=o_R(1),
\end{split}
\end{align}
where we also used the facts that
$$
\left(|s|^{q-2}s-|t|^{q-2}t \right)(s-t) \geq 0, \quad \forall \,\, s, t \in \R, q>1,
$$
$$
\int_{\R^N \backslash B(0, R)}|m_1|^{\frac Nq} \,dx=o_R(1),
$$
{where the second fact holds because of $m_1 \in L^{\frac N q}(\R^N)$ from the assumption $(H)$.} Combining \eqref{s1}, \eqref{ss2} and \eqref{s2}, by \eqref{compact}, we then obtain that
$$
\int_{\R^N} \left(a(x) \left(|\nabla u_n|^{p-2} \nabla u_n-|\nabla u|^{p-2} \nabla u \right)+\left(|\nabla u_n|^{q-2} \nabla u_n-|\nabla u|^{q-2} \nabla u \right)\right) \cdot \left(\nabla u_n-\nabla u\right)\,dx=o_n(1).
$$
Observe that
\begin{align} \label{monotone}
|z_1-z_2|^r \leq C \left(\left(|z_1|^{r-2} z_1 -|z_2|^{r-2}z_2\right) \cdot \left(z_1-z_2\right)\right)^{\frac {\theta}{2}} \left(|z_1|^r+|z_2|^r\right)^{1-\frac{\theta}{2}}, \quad \forall \, z_1, z_2 \in \R^N.
\end{align}
where $\theta=r$ if $1<r<2$ and $\theta=2$ if $r \geq 2$. Then we see that
\begin{align*}
&\int_{\R^N} a(x) \left(|\nabla u_n-\nabla u |^p\right) \, dx +\int_{\R^N} |\nabla u_n-\nabla u |^q\, dx \\
&\leq C \left(\int_{\R^N} a(x) \left(|\nabla u_n|^{p-2} \nabla u_n-|\nabla u|^{p-2} \nabla u \right)\cdot \left(\nabla u_n-\nabla u\right) \, dx \right)^{\frac {\theta}{2}} \left( \int_{\R^N} a(x) \left(|\nabla u_n|^p + |\nabla u|^p\right)\,dx \right)^{1-\frac{\theta}{2}}  \\
& \quad + C \left(\int_{\R^N} \left(|\nabla u_n|^{q-2} \nabla u_n-|\nabla u|^{q-2} \nabla u \right)\cdot \left(\nabla u_n-\nabla u\right) \, dx \right)^{\frac {\theta}{2}} \left( \int_{\R^N} |\nabla u_n|^q + |\nabla u|^q\,dx \right)^{1-\frac{\theta}{2}}=o_n(1).
\end{align*}
This immediately indicates that $u_n \to u$ in $D^{1, \xi}(\R^N)$ as $n \to \infty$. Taking advantage of \eqref{equ1} and \eqref{equ2}, we then get that
$$
\int_{\R^N} m(x)|u_n|^q \,dx=\int_{\R^N} m(x)|u|^q \,dx+o_n(1),
$$
because of $\lambda_n=\lambda+o_n(1)$ and $\lambda \neq 0$. In view of \eqref{convm1}, then
$$
\int_{\R^N} m_2(x)|u_n|^q \,dx=\int_{\R^N} m_2(x)|u|^q \,dx+o_n(1).
$$
Since $u_n \to u$ in $D^{1,q}(\R^N)$ as $n \to \infty$, by Hardy's inequality, then 
$$
\int_{\R^N} \frac{|u_n-u|^q}{(1+|x|)^q} \, dx\leq  \left(\frac{p}{N-p}\right)^p \int_{\R^N} |\nabla u_n-\nabla u|^q \, dx=o_n(1).
$$
Consequently, we derive that $u_n \to u$ in $E$ as $n \to \infty$. Thus the proof is completed.
\end{proof}

\begin{thm} \label{thm2}
Assume that $(H)$ holds, {$N \geq 2$}, $1<q<p<N$, ${a \in C^{0, 1}(\R^N, [0, +\infty))}$ and ${a \not\equiv 0}$. Then there exists a sequence of  solutions $(\lambda_k, u_k) \in \R \times E$ with $u_k \in \mathcal{S}$ and
$$
0<\lambda_1<\lambda_2 \leq \cdots \leq \lambda_k \leq \cdots, \quad \lim_{k \to \infty} \lambda_k \to +\infty  \,\, \mbox{as} \,\, k \to +\infty.
$$
\end{thm}
\begin{proof} 
To establish the existence of a sequence of eigenvalues to \eqref{eque}, we shall take into account Ljusternik-Schnirelman theory. Define
$$
\Sigma:=\left\{ A \subset \mathcal{S} : A \,\, \mbox{is compact and} \,\, A=-A\right\}.
$$
For a set $A \in \Sigma$, the genus of $A$ is defined by
$$
\gamma(A):=\min \left\{ n \in \mathbb{N} : \mbox{exists a function} \,\, \varphi \in C(A, \R^n \backslash \{0\}) \,\, \mbox{satisfying}\,\, \varphi(-x)=-\varphi(x) \right\}.
$$
If such a minimum does not exist, we set $\gamma(A)=+\infty$.

Let us now define 
$$
\Sigma_k:=\left\{ A \in \Sigma : \gamma(A) \geq k\right\}, \quad \forall \,\, k \in \mathbb{N}^+.
$$
First we see that, for any $k \in \mathbb{N}^+$, $\Sigma_k \neq \emptyset$. Indeed, let $X_k$ be a $k$ dimensional subspace of $E$, by Borsuk-Ulam's theorem, then $\gamma(\mathcal{S}\cap X_k) \geq k$. Define
$$
\tilde{\lambda}_k:=\inf_{A \in \Sigma_k} \sup_{u \in A} \Phi(u).
$$
Since $\Sigma_{k+1} \subset \Sigma_k$, then $\tilde{\lambda}_k \leq \tilde{\lambda}_{k+1}$ for any $k \in \mathbb{N}^+$. From Lemma \ref{psc}, then $\tilde{\lambda}_k$ is a critical points of $J$ restricted on $\mathcal{S}$ for any $k \in \mathbb{N}^+$. Then we derive that
$$
0<\tilde{\lambda}_1<\tilde{\lambda}_2 \leq \cdots \leq \tilde{\lambda}_k \leq \tilde{\lambda}_{k+1} \leq \cdots.
$$
Next we prove that $\tilde{\lambda}_k \to + \infty$ as $k \to + \infty$. Let $\{e_i\} \subset E$ be such that $E=\overline{\textnormal{span} \left\{e_1, e_2, \cdots, e_i, \cdots \right\}}$. Let $\{e_i'\} \subset E$ be such that $E'=\overline{\textnormal{span} \left\{e_1', e_2', \cdots, e_i', \cdots \right\}}$, where $E'$ denotes the dual space of $E$. Define $X_i:=\textnormal{span} \left\{e_i\right\}$ and
$$
Y_k:=\bigoplus_{i=1}^k X_i, \quad Z_k:=\overline{\bigoplus_{i=k}^{\infty}X_i}, \quad \forall \,\, k \in \mathbb{N}^+.
$$
Let $A \in \Sigma_k$ satisfy $\gamma(A) \geq k$. By basic properties of the genus, we have that $A \cap Z_k \neq \emptyset$. Define
$$
\beta_k:=\inf_{A \in \Sigma_k} \sup_{u \in A \cap Z_k} \Phi(u), \quad \forall \,\, k \in \mathbb{N}^+.
$$
Then $\beta_k \to + \infty$ as $k \to \infty$. Otherwise, we may assume that $\{\beta_k\} \subset \R$ is bounded. Thus there exists a sequence $\{u_k\} \subset A \cap Z_k$ such that $\{\Phi(u_k)\} \subset \R$ is bounded. It then follows that $\{u_k\}$ is bounded in $E$. Further, there exists $u \in E$ such that $u_k \wto u$ in $E$ as $n \to \infty$. Observe that $\langle e_i', u \rangle=\langle e_i', u_k\rangle + o_k(1)=o_k(1)$, because of $u_k \in Z_k$. Therefore, we have that $u=0$ and $u_k \wto 0$ in $E$ as $k \to \infty$. {This along with the assumption that $m_1 \in L^{\frac Nq}(\R^N)$ from the assumption $(H)$ leads to
$$
\int_{\R^N} m_1(x)|u_k|^q \,dx=o_k(1).
$$
Since $m_2 \geq 0$ from the assumption $(H)$, then
$$
\int_{\R^N} m(x)|u_k|^q \,dx=\int_{\R^N} m_1(x)|u_k|^q \,dx-\int_{\R^N} m_2(x)|u_k|^q \,dx \leq o_k(1),
$$
}
which is impossible due to $u_k \in \mathcal{S}$. Consequently, we get that $\beta_k \to +\infty$ as $k\to \infty$. Thanks to $\tilde{\lambda}_k \geq \beta_k$ for any $k \in \mathbb{N}^+$, then $\tilde{\lambda}_k \to +\infty$ as $k \to \infty$. Since $u_k \in E$ is a critical point for $\Phi$ restricted on $\mathcal{S}$, then there exists $\lambda_k \in \R$ such that
$$
-\Delta_p^a u_k-\Delta_q u_k = \lambda_k m(x)|u_k|^{q-2}u_k \quad \mbox{in} \,\, \R^N,
$$
where
$$
\lambda_k=\frac 1q \int_{\R^N} a(x)|\nabla u_k|^p\,dx + \frac 1q \int_{\R^N} |\nabla u_k|^q \, dx >\Phi(u_k)=\tilde{\lambda}_k, \quad \forall \,\, k \in \mathbb{N}^+.
$$
Thus the proof is completed.
\end{proof}

\begin{lem} \label{simple}
Define
\begin{align} 
\begin{split} \label{defI}
I(u, v):&=-\int_{\R^N} \left(\Delta_p^a u\right) \frac{u^q-v^q}{u^{q-1}} \, dx - \int_{\R^N} \left(\Delta_q u\right) \frac{u^q-v^q}{u^{q-1}} \, dx \\
& \quad -\int_{\R^N}\left(\Delta_p^a v\right) \frac{v^q-u^q}{v^{q-1}} \, dx -\int_{\R^N}\left(\Delta_q v\right) \frac{v^q-u^q}{v^{q-1}} \, dx, 
\end{split}
\end{align}
where $u, v \in D^{1, \xi}(\R^N)$, $u, v>0$ and $1<q<p$. Then $I(u, v) \geq 0$. Moreover, $I(u ,v)=0$ if and only if $u= kv$ for some $k \in \R$. 
\end{lem}
\begin{proof}
Let us first show that
$$
I_1(u,v):=-\int_{\R^N} \left(\Delta_p^a u\right) \frac{u^q-v^q}{u^{q-1}} \, dx -\int_{\R^N}\left(\Delta_p^a v\right) \frac{v^q-u^q}{v^{q-1}} \, dx \geq 0, \quad u, v \in D^{1, \xi}(\R^N), u ,v>0.
$$
It is straightforward to compute that
\begin{align} \label{c1}
\nabla \left(\frac{u^q-v^q}{u^{q-1}}\right)= \left(1+(q-1) \left(\frac vu\right)^q\right) \nabla u -q \left(\frac vu\right)^{q-1}\nabla v,
\end{align}
\begin{align} \label{c2}
\nabla \left(\frac{v^q-u^q}{v^{q-1}}\right)= \left(1+(q-1) \left(\frac uv\right)^q\right) \nabla v -q \left(\frac uv\right)^{q-1}\nabla u.
\end{align}
Therefore, by the divergence theorem, we derive that
\begin{align*}
I_1(u, v)&=\int_{\R^N} a(x)\left(\left(1+(q-1) \left(\frac vu\right)^q\right) |\nabla u|^p-q \left(\frac vu\right)^{q-1} |\nabla u|^{p-2} \nabla u \cdot \nabla v\right)\,dx \\
& \quad + \int_{\R^N} a(x)\left(\left(1+(q-1) \left(\frac uv\right)^q\right) |\nabla v|^p-q \left(\frac uv\right)^{q-1} |\nabla v|^{p-2} \nabla v \cdot \nabla u \right)\,dx.
\end{align*}
Using Young's inequality, we know that
\begin{align*}
q \left(\frac vu\right)^{q-1} |\nabla u|^{p-2} |\nabla u \cdot \nabla v| & \leq q \left(\frac vu\right)^{q-1} |\nabla u|^{p-1}|\nabla v| \\
&\leq\frac{q(p-1)}{p}\left(\frac vu\right)^{\frac{p(q-1)}{p-1}} |\nabla u|^p +\frac q p |\nabla v|^{p} \\
& = \frac{q(p-1)}{p}\left(\frac vu\right)^{\frac{p(q-1)}{p-1}} |\nabla u|^{\frac{p^2(q-1)}{q(p-1)}}|\nabla u|^{\frac{p(p-q)}{q(p-1)}} +\frac q p |\nabla v|^{p} \\
& \leq (q-1)\left(\frac vu\right)^q |\nabla u|^p +\frac{p-q}{p}|\nabla u|^p + \frac q p |\nabla v|^{p}.
\end{align*}
Similarly, we can get that
\begin{align*}
q \left(\frac uv\right)^{q-1} |\nabla v|^{p-2} |\nabla v \cdot \nabla u| \leq q \left(\frac uv\right)^{q-1} |\nabla v|^{p-1}|\nabla u| 
 \leq (q-1)\left(\frac uv\right)^q |\nabla v|^p +\frac{p-q}{p}|\nabla v|^p+ \frac q p |\nabla u|^{p}.
\end{align*}
It then follows that $I_1(u, v) \geq 0$. 
Next we prove that
$$
I_2(u, v):=- \int_{\R^N} \left(\Delta_q u\right) \frac{u^q-v^q}{u^{q-1}} \, dx-\int_{\R^N}\left(\Delta_q v\right) \frac{v^q-u^q}{v^{q-1}} \, dx \geq 0, \quad u, v \in D^{1, \xi}(\R^N), u ,v>0.
$$
In view of \eqref{c1} and \eqref{c2}, by the divergence theorem, then
\begin{align*}
I_2(u ,v)&=\int_{\R^N} \left(1+(q-1) \left(\frac vu\right)^q\right) |\nabla u|^q-q \left(\frac vu\right)^{q-1} |\nabla u|^{q-2} \nabla u \cdot \nabla v \,dx \\
& \quad + \int_{\R^N} \left(1+(q-1) \left(\frac uv\right)^q\right) |\nabla v|^q-q \left(\frac uv\right)^{q-1} |\nabla v|^{q-2} \nabla v \cdot \nabla u \,dx.
\end{align*}
Using again Young's inequality, we obtain that
$$
q \left(\frac vu\right)^{q-1} |\nabla u|^{q-2} |\nabla u \cdot \nabla v| \leq q \left(\frac vu\right)^{q-1} |\nabla u|^{q-1} |\nabla v| \leq (q-1)\left(\frac vu\right)^q |\nabla u|^q + |\nabla v|^q,
$$
$$
q \left(\frac uv\right)^{q-1} |\nabla v|^{q-2} |\nabla v \cdot \nabla u| \leq q \left(\frac uv\right)^{q-1} |\nabla v|^{q-1} |\nabla u| \leq (q-1)\left(\frac uv\right)^q |\nabla v|^q + |\nabla u|^q.
$$
Therefore, we have that $I_2(u ,v)=0$. Accordingly, there holds that $I(u ,v) \geq 0$ for any $u ,v \in D^{1, \xi}(\R^N)$ and $u, v>0$. If $I(u, v)=0$, then $I_2(u ,v)=0$. This leads to
$$
\nabla u \cdot \nabla v=|\nabla u| |\nabla v|, \quad \left(\frac vu \right)^q |\nabla u|^q=|\nabla v|^q, \quad \left(\frac uv\right)^q |\nabla v|^q=|\nabla u|^q,
$$
As a consequence, we see that
$$
| u \nabla v-v\nabla u|=0.
$$
This implies that there exists $k \in \R$ such that $u=k v$ and the proof is completed.
\end{proof}

\begin{rem}\label{rem11}
{In fact, Lemma \ref{simple} is established for the double phase operator under the assumption $q<p$, which is not a direct consequence of Lemma \ref{Iuv1}. It is unknown to us if Lemma \ref{simple} remains valid for the case $p<q$. From the proof of Lemma \ref{simple}, one can see that the assumption $q<p$ is crucial, which is the premise of the use of Young's inequality.}
\end{rem}

\begin{prop} \label{simple1}
Assume that $(H)$ holds, {$N \geq 2$}, $1<q<p<N$, ${\frac{p}{q}<1+\frac 1 N}$, ${a \in C^{0, 1}(\R^N, [0, +\infty))}$ and $a \not\equiv 0$. Assume that any eigenfunction to \eqref{eque} corresponding to $\lambda$ is nonnegative. Then $\lambda$ is simple.
\end{prop}
\begin{proof}
Let $u \in E$ be a nonnegative eigenfunction to \eqref{eque} corresponding to $\lambda$.
It follows from \cite{CS} and \cite[Proposition 3]{PRZ} or \cite[Proposition 2.3]{PVV} that $u>0$. Let $u>0$ and $v>0$ be two eigenfunctions to \eqref{eque} corresponding to $\lambda$. 
Then we see that 
$$
-\Delta_p^a u-\Delta_q u =\lambda m(x)u^{q-1}, \quad -\Delta_p^a v-\Delta_q v =\lambda m(x)v^{q-1} \quad \mbox{in} \,\, \R^N.
$$
As a result, there holds that
$$
I(u, v)= \lambda \int_{\R^N} m(x) \left(u^q-v^q\right) \,dx+ \lambda \int_{\R^N} m(x) \left(v^q-u^q\right) \,dx=0.
$$
It then follows from Lemma \ref{simple} that the desired conclusion holds. This completes the proof.
\end{proof}


\begin{prop} \label{se}
Assume that $(H)$ holds, {$N \geq 2$}, {$1<p, q<N$, $p \neq q$}, ${a \in C^{0, 1}(\R^N, [0, +\infty))}$ and $a \not\equiv 0$. Then 
$$
\mu_{1,q,1}=\inf \left\{\frac{\frac 1p \int_{\R^N} a(x)|\nabla u|^p\,dx + \frac 1q \int_{\R^N} |\nabla u|^q \, dx}{\frac 1q \int_{\R^N} m(x) |u|^q \,dx} : u \in E \backslash \{0\},\, \int_{\R^N} m(x) |u|^q \,dx>0\right\}.
$$
\end{prop}
\begin{proof}
Since $\mu_{1,q,1}$ is the first eigenvalue to \eqref{eque1} and $E \subset D^{1,q}(\R^N)$, then
\begin{align*}
\mu_{1,q,1}&=\inf \left\{\frac{\int_{\R^N} |\nabla u|^q \, dx}{\int_{\R^N} m(x) |u|^q \,dx} : u \in  D^{1,q}(\R^N) \backslash \{0\}, \, \int_{\R^N} m(x) |u|^q \,dx>0\right\} \\
 &\leq \inf \left\{\frac{\frac 1p \int_{\R^N} a(x)|\nabla u_{1,q,1}|^p\,dx + \frac 1q \int_{\R^N} |\nabla u|^q \, dx}{\frac 1q \int_{\R^N} m(x) |u|^q \,dx} : u \in E \backslash \{0\}, \, \int_{\R^N} m(x) |u|^q \,dx>0\right\}.
\end{align*}
Let $u_{1,q,1} \in E$ be an eigenfunction to \eqref{eque1} corresponding to $\mu_{1,q,1}$ and $p<q$, then
\begin{align*}
&\inf \left\{\frac{\frac 1p \int_{\R^N} a(x)|\nabla u|^p\,dx + \frac 1q \int_{\R^N} |\nabla u|^q \, dx}{\frac 1q \int_{\R^N} m(x) |u|^q \,dx} : u \in E \backslash \{0\}, \, \int_{\R^N} m(x) |u|^q \,dx>0\right\} \\
&\leq  \frac{\frac{n^p}{p}\int_{\R^N} a(x)|\nabla u_{1,q,1}|^p\,dx +\frac{n^q}{q}\int_{\R^N} |\nabla u_{1,q,1}|^q \, dx}{\frac{n^q}{q} \int_{\R^N} m(x)|u_{1,q,1}|^q \,dx}=\mu_{1,q,1} +o_n(1) \quad \mbox{as} \,\, n \to \infty.
\end{align*}
Similarly, if $q<p$, then
\begin{align*}
&\inf \left\{\frac{\frac 1p \int_{\R^N} a(x)|\nabla u|^p\,dx + \frac 1q \int_{\R^N} |\nabla u|^q \, dx}{\frac 1q \int_{\R^N} m(x) |u|^q \,dx} : u \in E \backslash \{0\}, \, \int_{\R^N} m(x) |u|^q \,dx>0\right\} \\
&\leq  \frac{\frac{1}{pn^p}\int_{\R^N} a(x)|\nabla u_{1,q,1}|^p\,dx +\frac{1}{qn^q}\int_{\R^N} |\nabla u_{1,q,1}|^q \, dx}{\frac{1}{qn^q} \int_{\R^N} m(x)|u_{1,q,1}|^q \,dx}=\mu_{1,q,1} +o_n(1) \quad \mbox{as} \,\, n \to \infty.
\end{align*}
Thus the desired result follows and the proof is completed.
\end{proof}

\begin{rem}
Under the assumptions of Theorem \ref{thm2}, by Theorem \ref{nonexist} and Proposition \ref{se}, we have
$$
\lambda_1>\inf \left\{\frac{\frac 1p \int_{\R^N} a(x)|\nabla u|^p\,dx + \frac 1q \int_{\R^N} |\nabla u|^q \, dx}{\frac 1q \int_{\R^N} m(x) |u|^q \,dx} : u \in E \backslash \{0\},\, \int_{\R^N} m(x) |u|^q \,dx>0\right\}.
$$
\end{rem}

\begin{rem} The arguments developed in this paper allow to obtain similar results if the hypothesis $(H)$ is replaced by the following condition introduced by Szulkin and Willem \cite{p4},
\begin{itemize}
\item[$(\mathcal{H})$]  $m \in L^{1}_{loc}(\R^N)$, $m^+=m_1+m_2\not=0$, $m_1 \in L^{\frac N q}(\R^N)$, for every $y\in\R^N$, $\lim_{x\rightarrow y}|x-y|^qm_2(x)=0$ and $\lim_{|x|\rightarrow\infty}|x|^qm_2(x)=0$, where $m^+:=\max\{m(x), 0\}$.
\end{itemize}
\end{rem}


\end{document}